\documentclass[12pt,reqno]{amsart}

\usepackage{amsmath}
\usepackage{amsfonts}
\usepackage{amsthm}
\usepackage{amstext}
\usepackage{amssymb}
\usepackage{txfonts}
\usepackage{enumerate}
\usepackage{graphicx,color}
\usepackage{setspace}
\usepackage{epsfig}
\usepackage{color}

\usepackage{hyperref}
\usepackage{verbatim}
\theoremstyle{plain}
\newtheorem{theorem}{Theorem}

\theoremstyle{definition}
\newtheorem{definition}{Definition}
\newtheorem{remark}{Remark}
\newtheorem{exm}{Example}

\begin{document}

\title[Monotonicity]{Cone Monotonicity: structure theorem, properties,
and comparisons to other notions of Monotonicity}

\author[H. A. Van Dyke]{Heather A. Van Dyke}
\email{hvandyke@math.wsu.edu}
\author[K. R. Vixie]{Kevin R. Vixie}
\email{vixie@speakeasy.net}
\author[T. J. Asaki]{Thomas J. Asaki}
\email{tasaki@wsu.edu}
\address{Department of Mathematics \\
Washington State University \\
Pullman, WA 99164-3113}

\date{\today}
\begin{abstract}
In search of a meaningful 2-dimensional analog to monotonicity, we introduce two new definitions and give examples of and discuss the relationship between these definitions and others that we found in the literature.

Note: After we published the article in \emph{Abstract and
Applied Analysis} and after we searched multiple times
for previous work, we discovered that Clarke at al. had
introduced the definition of cone monotonicity and given
a characterization. See the addendum at the end of this
paper for full reference information.
\end{abstract}

\maketitle

%%%%%%%%%%%%%%%%%%%%%%%%%%%%%%%%%%%%%%%%%%%%%%%
%%%Introduction
%%%%%%%%%%%%%%%%%%%%%%%%%%%%%%%%%%%%%%%%%%%%%%%
\begin{section}{Introduction}
Though monotonicity for functions from $\mathbb{R}$ to $\mathbb{R}$ is familiar in even the most elementary courses in mathematics, there are a variety of definitions in the case of functions from $\mathbb{R}^n$ to $\mathbb{R}$.  In this paper we review the definitions we found in the literature and suggest a new definition (with its variants) which we find useful.
 
In Section \ref{old}, we introduce the definitions from the literature for $n$ dimensional monotone functions ($n\geq 2$) .  We give examples and discuss the relationship between these definitions.  In Section \ref{New}, we introduce a new definition of monotonicity\footnote{After we published the article in \emph{Abstract and
Applied Analysis} and after we searched multiple times
for previous work, we discovered that Clarke at al. had
introduced the definition of cone monotonicity and given
a characterization. See the addendum at the end of this
paper for full reference information.} and some of its variants.  We then give examples and explore the characteristics of these new definitions.
\end{section}

%%%%%%%%%%%%%%%%%%%%%%%%%%%%%%%%%%%%%%%%%%%%%%%
%%%Definitions and Examples
%%%%%%%%%%%%%%%%%%%%%%%%%%%%%%%%%%%%%%%%%%%%%%%
\begin{section}{Definitions and Examples}\label{old}
In \cite{Lebesgue1907}, Lebesgue notes that on any interval in $\mathbb{R}$ a monotonic function $f$ attains its maximum and minimum at the endpoints of this interval.  This is the motivation he uses to define monotonic functions on an open bounded domain, $\Omega\subset\mathbb{R}^2$.  His definition requires that these functions must attain their maximum and minimum values on the boundaries of the closed subsets of $\Omega$.  We state the definition found in \cite{Lebesgue1907} below.
\begin{definition}[Lebesgue]\label{LebMono}
Let $\Omega$ be an open bounded domain.  A continuous function $f:\Omega\subset\mathbb{R}^2\rightarrow\mathbb{R}$ is said to be Lebesgue monotone if in every closed domain, $\Omega^\prime\subseteq\Omega$, $f$ attains its maximum and minimum values on $\partial\Omega^\prime$.
\end{definition}
\begin{remark}\label{LebNoLocEx}
This definition tells us that a nonconstant function $f$ is Lebesgue monotone if and only if no level set of $f$ is a local extrema.
\end{remark}
\begin{remark}  Notice also that we can extend this definition to a function $f:\Omega\subset\mathbb{R}^n\rightarrow\mathbb{R}$.
\end{remark}
We now give a couple examples of functions that are Lebesgue monotonic.  

\begin{exm} Since an $n$ dimensional plane, $f(x)=c^Tx+x_0$, can only take on extreme values on the boundary of any closed set in its domain, we know that it is Lebesgue Monotone.
%It is easily seen that a function whose graph is an $n$ dimensional plane (i.e. $f(x)=c^Tx+x_0$) is Lebesgue monotone since it does not have any local extrema, it can only take on extreme values on the boundary of any closed set in its domain. 
\end{exm}

\begin{exm}\label{2Dnonmonocubic}
Let $\Omega=R(x,L)$ be the square of side length $L$, centered at a point $x\in\mathbb{R}^n$, for some $L>0$.  Any function of $n$ real variables whose level sets are lines is Lebesgue monotone.  For example, let $f(x,y)=x^3-x$ (see Figure \ref{cubicMono}).  
\begin{figure}[!htbp]
\centering
\includegraphics{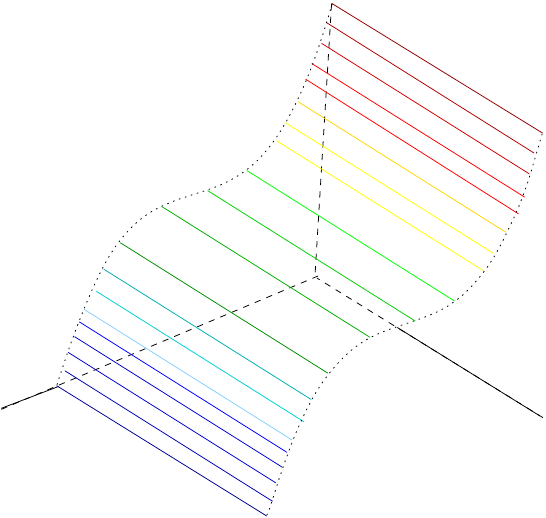}
\caption{$f(x,y)=x^3-x$}\label{cubicMono}
\end{figure}
Because the function is constant in the $y$ direction, we see that on the boundary of any closed subset of $\Omega$, $f$ must take on all the same values as it takes in the interior.  Of course, the choice of $\Omega$ is somewhat arbitrary here (it need only be bounded).
\end{exm}

We now move on to another definition given in \cite{Mostow1968}.  Here Mostow,  gives the following definition for monotone functions.
\begin{definition}[Mostow]\label{MosMono}
Let $\Omega$ be an open set in a locally connected topological space and let $f$ be a continuous function on $\overline{\Omega}$.  The function $f$ is called Mostow monotone on $\Omega$ if for every connected open subset $U\subset \Omega$ with $U\neq\overline{U}$, 
\begin{equation*}
\sup_{x\in U} f(x)\leq\sup_{y\in\partial U}f(y)\quad\mbox{and}\quad\inf_{x\in U}f(x)\geq\inf_{y\in\partial U}f(y). 
\end{equation*}
\end{definition}
We see that if $\Omega=\mathbb{R}^2$ then we can choose a closed disk, $D_r=D(0,r)$ centered at the origin with radius $r$ so that $U=\mathbb{R}^2\setminus D_r$.  On $\partial U=\partial D_r$ a function, $f$, that is Mostow monotone must obtain both its maximum and its minimum.  But, we can let $r\searrow 0$.  In doing this, we see that the maximum and minimum of $f$ can be arbitrarily close.  This tells us that if $f$ is Mostow Monotone, then it must be a constant function.  In \cite{Mostow1968}, Mostow states that one can adjust this definition by requiring the function to take on its maximum or minimum on $\partial U$ only for relatively compact open sets.

%%%%%%%%%%%%%%%%%%%%%%%%%%%%%%%%%%%%%%%%%%%%
% M-Monotone ===> L-Monotone
%%%%%%%%%%%%%%%%%%%%%%%%%%%%%%%%%%%%%%%%%%%%
\begin{exm}
It is not true that Lebesgue monotone functions are Mostow monotone (even if we follow the suggestion in \cite{Mostow1968} to adjust the definition of Mostow monotone).  To see this, we consider a function $f:\Omega\subset\mathbb{R}^2\rightarrow\mathbb{R}$ that is affine and has its gradient oriented along the domain as in Figure \ref{LebNotMost}.  Here $f$ will have supremum and infimum that are not attained on the boundary of the open set $U$. 
\begin{figure}[htbp!]
\centering
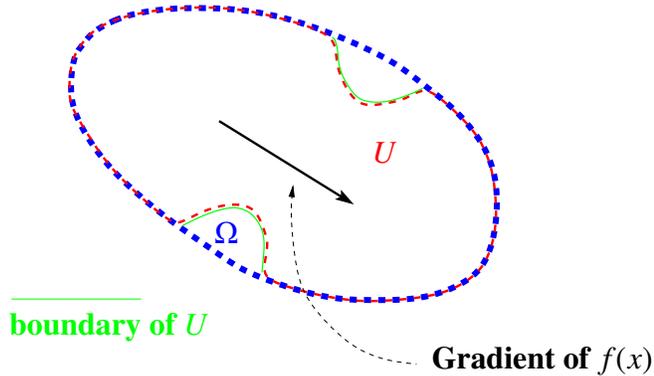
\caption{Example of a function which is Lebesgue monotone, but not Mostow monotone}\label{LebNotMost}
\end{figure}
\end{exm}
\begin{remark}\label{MimpliesL}
Notice, if $\Omega\subset\mathbb{R}^2$ is a bounded domain then any continuous, Mostow monotone function is also Lebesgue monotone.  This is true whether or not we are adjusting the definition as suggested in \cite{Mostow1968}.
\end{remark}

Before giving the next definition, we give some notation for clarity.  Let $\Omega\subseteq\mathbb{R}^2$ be an open domain, $B(x,r)$ be the closed ball of radius $r$ around the point $x\in \Omega$, and $S(x,r)$ be the boundary of the ball, $B(x,r)$.  We say a function is $L^1_{loc}(\Omega)$ if $\int_U|u|\;dx<\infty$ for every bounded set $U\subset\Omega$.  For comparison, we write the following definition for a less general function than what can be found in \cite{VG1976}.

\begin{definition}[Vodopyanov, Goldstein]\label{VGMono}
We say an $L^1_{loc}$ function, $f:\Omega\rightarrow \mathbb{R}$ is Vodopyanov Goldstein Monotone at a point $x\in \Omega$ if there exists $0<r(x)\leq \mbox{dist}(x,\partial\Omega)$ so that for almost all $r\in[0,r(x)]$, the set 
\begin{equation*}
f^{-1}\left(f(B(x,r))\cap [\mathbb{R}\setminus f(S(x,r))]\right)\cap B(x,r)
\end{equation*} 
has measure zero.  A function is then said to be Vodopyanov-Goldstein monotone on a domain, $\Omega$ if it is Vodopyanov Goldstein monotone at each point $x\in\Omega$.
\end{definition}

\begin{exm}
If we remove the continuity requirement for both Lebesgue and Mostow monotone functions we can create a function that is Mostow monotone but not Vodopyanov-Goldstein monotone.  For the function in Figure \ref{discontLM}, we see that any closed and bounded set must attain both the maximum and minimum of $f$ on its boundary, but if we take a ball, $B$ that contains the set $\{f=0\}$, we see that $f(S)=\{-1,1\}$. So, $f^{-1}(f(B\cap\mathbb{R}\setminus f(S)))\cap B$ does not have measure zero.  That is, $f$ is not Vodopyanov-Goldstein monotone.
\begin{figure}[!htbp]
\centering
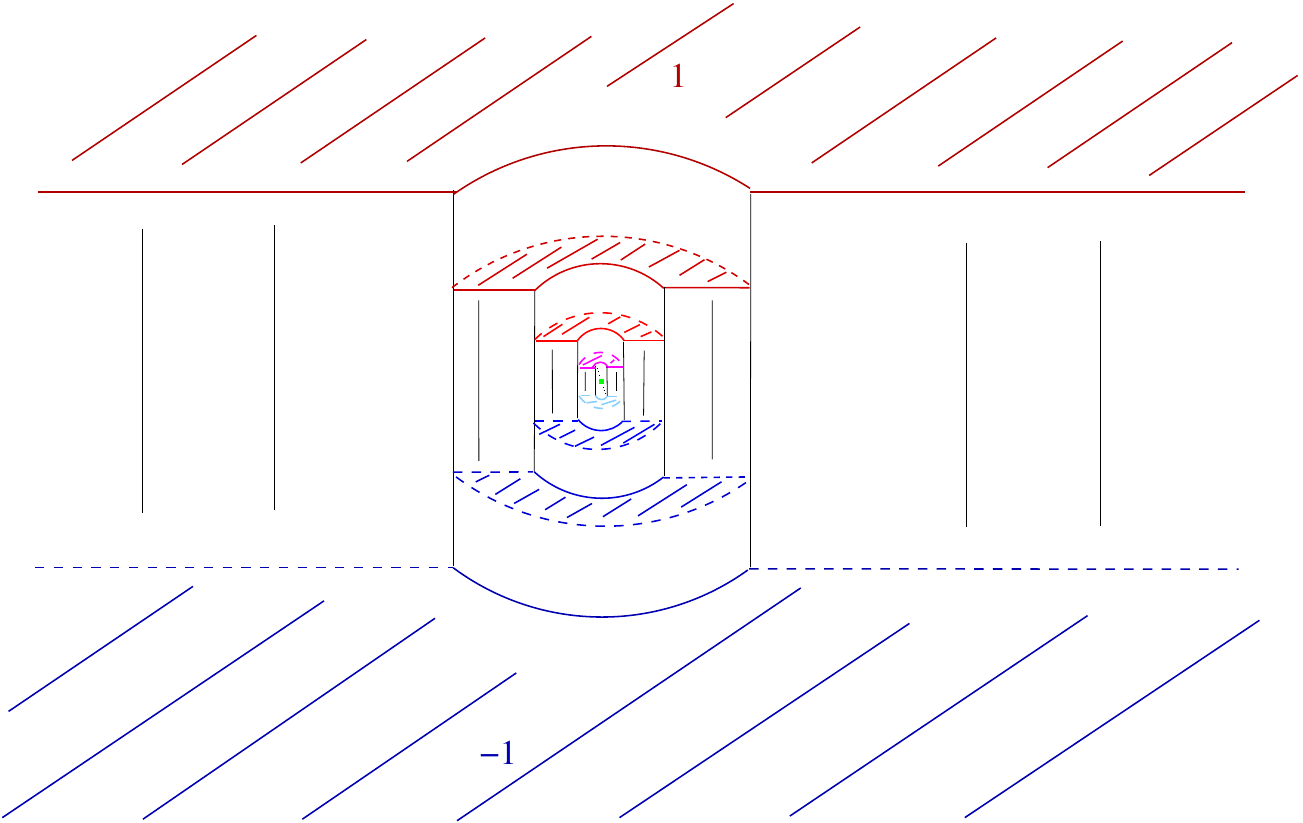
\caption{Function satisfying all but continuity criteria for Mostow Monotone and is not Vodopyanov-Goldstein monotone.}\label{discontLM}
\end{figure}
\end{exm}

\begin{exm}\label{VGNotL}Now, a function can be Vodopyhanov-Goldstein monotone, but not Lebesgue monotone.  An example of such a function is one in which $f$ attains a minimum along a set, $\mathcal{M}$, that is long and narrow relative to the set $\Omega$ (see Figure \ref{VGdniL}).  In this case, the boundary of any ball, $B(x,r)\subset\Omega$, that is centered along this set must intersect the set, $\mathcal{M}$ thus attaining both its maximum and minimum on the boundary of the ball, but the function will not reach its minimum on the boundary of a closed set $\Omega^\prime$ such as the one in Figure \ref{VGdniL}.
\end{exm}

\begin{figure}[!htbp]
\centering
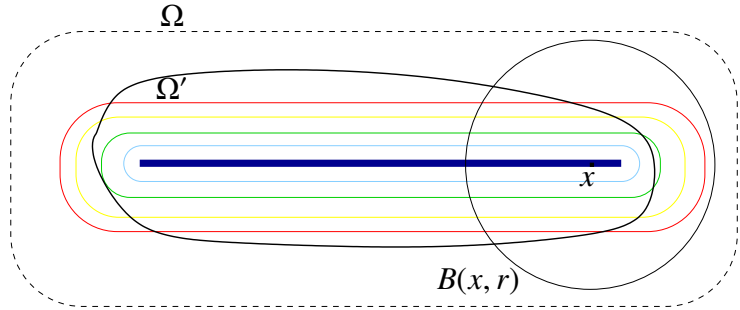
\caption{The level sets of a function that is Vodopyhanov-Goldstein monotone but not Lebesgue monotone.}\label{VGdniL}
\end{figure}

The next theorem shows that, for continuous functions, Lebesgue monotone functions are Vodopyanov-Goldstein monotone.
%%%%%%%%%%%%%%%%%%%%%%%%%%%%%%%%%%%%%%%%%%%%
% VG-Monotone <== L-Monotone
%%%%%%%%%%%%%%%%%%%%%%%%%%%%%%%%%%%%%%%%%%%%
\begin{theorem}\label{LMimpliesVGM}
Let $\Omega\subset\mathbb{R}^2$ be a bounded domain and let $f:\Omega\rightarrow\mathbb{R}$ be continuous.  Then $f$ is Vodopyanov-Goldstein monotone function if $f$ is Lebesgue monotone.
\end{theorem}
\begin{proof}  Suppose $f$ is Lebesgue monotone, then we know that for all closed sets $\Omega^\prime\subset\Omega$, $f$ attains its local extrema on  $\partial\Omega^\prime$.  In particular, if we let $x\in\Omega$, we have that $f$ attains its local extrema on the boundary of $B(x,r)$ for any $r>0$.  Let $M$ and $m$ be such that
\begin{equation*}
M\equiv\sup_{y\in B(x,r)}f(y)\quad\mbox{and}\quad m\equiv\inf_{y\in B(x,r)}f(y).
\end{equation*}
Then we know that $f(B(x,r))=(m,M)$ and $f(S(x,r))=[m,M]$.  So
\begin{eqnarray*}
\mathbb{R}\setminus f(S(x,r))=(-\infty,m)\cup(M,\infty)\\
\Rightarrow\quad f(B(x,r))\cap [(-\infty,m)\cup(M,\infty)]=\emptyset.
\end{eqnarray*}
Thus, 
\begin{equation*}
f^{-1}\left( f(B(x,r))\cap[(-\infty,m)\cup(M,\infty)]\right)=\emptyset.
\end{equation*}
So, the measure of the set
\begin{equation*}
B(x,r)\cap f^{-1}\left( f(B(x,r))\cap[(-\infty,m)\cup(M,\infty)]\right)
\end{equation*}
is zero.  Thus, $f$ is Vodopyanov Goldstein monotone at $x$.  Since $x$ was chosen arbitrarily, $f$ is Vodopyanov Goldstein monotone.
\end{proof}

In \cite{Manfredi1994}, Manfredi gives a definition for weakly monotone functions.
\begin{definition}[Manfredi]\label{WkMono}
Let $\Omega$ be an open set in $\mathbb{R}^n$ and $f:\Omega\rightarrow\mathbb{R}$ be a function in $W^{1,p}_{loc}(\Omega)$.  We say that $u$ is weakly monotone if for every relatively compact subdomain $\Omega^\prime\subset\Omega$ and for every pair of constants $m\leq M$ such that 
\begin{equation*}
(m-f)^+\in W^{1,p}_0(\Omega^\prime)\quad \mbox{and}\quad (f-M)^+\in W^{1,p}_0(\Omega^\prime),
\end{equation*}
we have that 
\begin{equation*}
m\leq f(x)\leq M\quad \mbox{for a.e. }x\in\Omega^\prime.
\end{equation*}
\end{definition}

Manfredi also gives the following example of a function that is weakly monotone, but not continuous (in this case at the origin).
\begin{exm}[Manfredi]
Write $z=re^{i\theta}$ for $z\in\mathbb{R}^2$.  Define $u$ by
\begin{equation*}
f(z)=\left\{\begin{array}{rl}
\theta & \mbox{for } 0\leq \theta\leq\pi/2,\\
\pi/2 & \mbox{for }\pi/2\leq\theta\leq\pi,\\
3\pi/2-\theta& \mbox{for }\pi\leq\theta\leq 3\pi/2,\\
0 & \mbox{for } 3\pi/2\leq\theta\leq 2\pi.
\end{array}\right.
\end{equation*}
\end{exm}
We expect that all the above types of monotone functions should be weakly monotone.  Because this function is not continuous, it does not satisfy the definition of Lebesgue or Mostow montone.

%%%%%%%%%%%%%%%%%%%%%%%%%%%%%%%%%%%%%%%%%%%%
% L-Monotone ==> Weak Monotone
%%%%%%%%%%%%%%%%%%%%%%%%%%%%%%%%%%%%%%%%%%%%
\begin{theorem}\label{LMimpliesWM}
Let $\Omega\subset\mathbb{R}^2$ be a bounded domain and $u:\Omega\rightarrow\mathbb{R}$, if $u$ is Lebesgue monotone, then $u$ is weakly monotone. 
\end{theorem}
\begin{remark}
Using Theorem \ref{LMimpliesWM} and Remark \ref{MimpliesL}, we see that a function that is Mostow Monotone is also Weakly Monotone.
\end{remark}
\begin{proof}
Let $\Omega^\prime\subset\Omega$, then by Definition \ref{LebMono}, $u$ is continuous and $u$ attains its maximum and minimum on $\partial\Omega^\prime$.  Let $m,M$ be a pair so that 
\begin{equation}\label{cutoffs}
(m-u)^+,(u-M)^+\in W^{1,p}_0(\Omega^\prime).
\end{equation}
Since $u$ is continuous so are $(m-u)^+$ and $(u-M)^+$.  Thus, (\ref{cutoffs}) gives us that 
\begin{equation*}
m\leq u\leq M \quad \mbox{ on }\partial\Omega^\prime.
\end{equation*}
Thus, $m\leq\min_{x\in\Omega^\prime}u(x)\leq u\leq \max_{x\in\Omega^\prime}u(x)\leq M$.  Thus, $u$ is weakly monotone.
\end{proof}
\end{section}

%%%%%%%%%%%%%%%%%%%%%%%%%%%%%%%%%%%%%%%%%%%%
% Normal, Cone, K Monotone Section
%%%%%%%%%%%%%%%%%%%%%%%%%%%%%%%%%%%%%%%%%%%%
\begin{section}{Normal Monotone, Cone Monotone, and $K$ Monotone}\label{New}
In this section, we introduce a new definition of monotonicity which we call \emph{Cone monotone}.  We will discuss some variants of this new definition that we call \emph{Normal monotone} and \emph{$K$ monotone}.  We also characterize $K$ monotone functions.

%%%%%%%%%%%%%%%%%%%%%%%%%%%%%%%%%%%%%%%%%%%%%%%
%%% Cone Monotone
%%%%%%%%%%%%%%%%%%%%%%%%%%%%%%%%%%%%%%%%%%%%%%%
\begin{subsection}{Cone Monotone}
Motivated by the notion of monotone operators, we give a more general definition of monotonicity for functions in 2 dimensions.  But first, we define the partial ordering, $\leq_K$ on $\mathbb{R}^2$.
\begin{definition}
Given a convex cone, $K\subset\mathbb{R}^2$ and two points $x,y\in\mathbb{R}^2$, we say that 
\begin{equation}
x\leq_K y\quad\mbox{if}\quad y-x\in K.
\end{equation}
\end{definition}
\begin{definition}
We say a function $f:\Omega\subseteq\mathbb{R}^2\rightarrow\mathbb{R}$ is cone monotone if at each $x\in\Omega$ there exists a cone, $K(x)$, so that 
\begin{equation}
f(x)\leq f(y)\quad\mbox{whenever}\quad x\leq_{K(x)} y.
\end{equation}
We say a function is $K$ monotone if the the function is cone monotone with a fixed cone $K$.
\end{definition}
\end{subsection}
%%%%%%%%%%%%%%%%%%%%%%%%%%%%%%%%%%%%%%%%%%%%%%%
%%%Characterization of Cone Monotone
%%%%%%%%%%%%%%%%%%%%%%%%%%%%%%%%%%%%%%%%%%%%%%%

\begin{subsubsection}{Characterization of Cone Monotone}\label{Characteristics}
 Here we first notice that a function that is $K$ monotone cannot have any local extrema.  This is stated more precisely in the following
 \begin{theorem}\label{nolocextr}
Assume $K$ is a convex cone with non-empty interior.  If $f$ is $K$ monotone then there is no compact connected set $M$ so that $f(M)$ is a local extremum.
%, that is, for every point $x\in M$, $f(x)\geq f(y)$ for all $y\in M_\varepsilon$, a neighborhood of $M$.
 \end{theorem}
 \begin{proof}
 \begin{figure}[htbp]
\centering
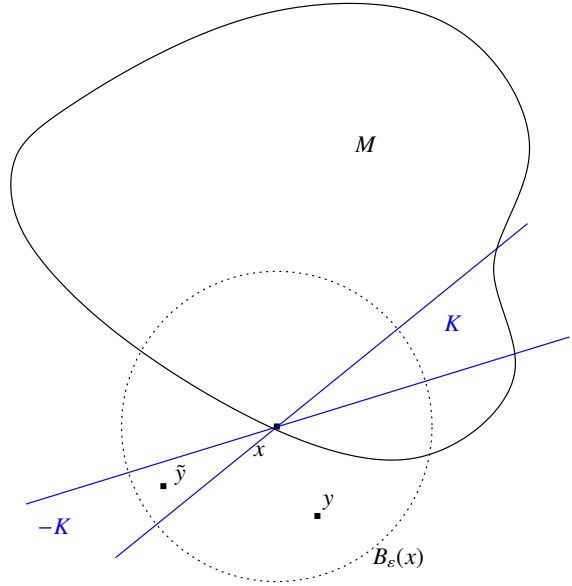
\caption{Cone monotone functions have no local extrema.}\label{nolocmin}
\end{figure}
 
Suppose to the contrary.  That is, suppose that $f(M)$ is a local minimum and suppose $f$ is $K$ monotone.  Then we have for every point $x\in\partial{M}$ and every $y\in B_\varepsilon(x)\setminus M,$ that $f(x)<f(y)$ (see Figure \ref{nolocmin}).

Pick $x\in\partial M$ so that the set $\{y\in M |x\leq_K y\}\neq\emptyset$.  We then consider the cone $- K=\{-x|x\in K\}$.  We know that if $\tilde{y}\in B_\varepsilon(x)\setminus M$ and $\tilde{y}-x\in-K$ then $x-\tilde{y}\in K$ so $f(x)\geq f(\tilde{y})$. Thus, we have a contradiction.
 \end{proof}
 \begin{remark}
Theorem \ref{nolocextr} and Remark \ref{LebNoLocEx} give us that a continuous $K$ monotone function is also Lebesgue monotone.
 \end{remark}
 For the following discussion, we work in the graph space, $\mathbb{R}^{n+1}$ of a $K$ monotone function $f:\mathbb{R}^n\rightarrow\mathbb{R}$.  Assume a fixed closed, convex cone, $K$ with non-empty interior.  Set
 \begin{eqnarray*}
 \overline{K}=&K\times(-\infty,0]\subset\mathbb{R}^{n+1}\\
  \underline{K}=&-K\times[0,\infty)\subset\mathbb{R}^{n+1}.
 \end{eqnarray*}
 Let $\vec{x}$ denote the vector $(x_1,x_2,...,x_n)$.   We can translate these sections up to the graph of $f$ so that it touches at the point $(\vec{x},f(\vec{x}))$.  In doing this we see that we have (see Figure \ref{extendedcones})
 \begin{eqnarray}
 \overline{K}+(\vec{x},f(\vec{x}))\subset\left\{(\vec{x},x_{n+1})|x_{n+1}\leq f(\vec{x})\right\}\nonumber\\
 \underline{K}+(\vec{x},f(\vec{x}))\subset\left\{(\vec{x},x_{n+1})|x_{n+1}\geq f(\vec{x})\right\}.
 \end{eqnarray}
 \begin{figure}[htbp!]
\centering
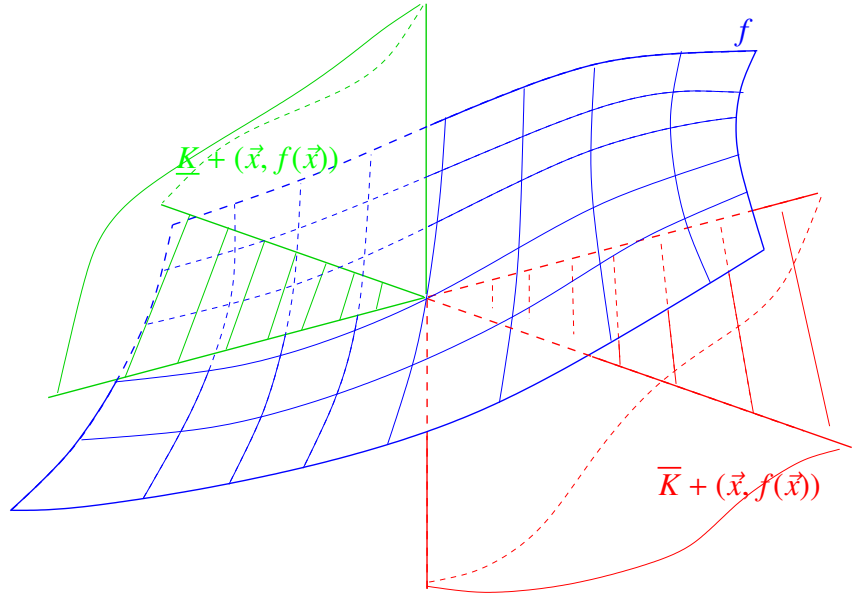
\caption{Example of $\overline{K}+(\vec{x},f(\vec{x}))$ and $\underline{K}+(\vec{x},f(\vec{x}))$.}\label{extendedcones}
\end{figure}
 
 We can do this for each point $(\vec{x},f(\vec{x}))$ on the graph of $f$.  Thus, the boundary of the epigraph and the boundary of the epograph are the same where we touch $\partial\mbox{epi} f$ with a translated $\overline{K}$ and $\underline{K}$. 
So, we can take the union of all such points to get 
 \begin{eqnarray}\label{touchingf}
cl(\mbox{epi} f) =\overline{\bigcup_{\vec{x}\in\mathbb{R}^n}\underline{K}+(\vec{x},f(\vec{x}))}\nonumber\\
cl(\mbox{epo} f)=\overline{\bigcup_{\vec{x}\in\mathbb{R}^n}\overline{K}+(\vec{x},f(\vec{x}))}.
 \end{eqnarray}
 Care needs to be taken in the case when $f$ has a jump discontinuity at $\vec{x}$.  Since for example, for an upper semicontinuous function $\mbox{epi}f$ does not contain points along the vertical section, $\{(\vec{x},r)|r\leq f(\vec{x})\}$, below the point $(\vec{x},f(\vec{x}))$. 
Let 
 \begin{equation}
\mathcal{E}= \bigcup_{\vec{x}\in\mathbb{R}^n}\underline{K}+(\vec{x},f(\vec{x})).
 \end{equation}
Using a limiting argument we notice that indeed this vertical section is contained in $\mathcal{E}$. If $(\vec{x},r)\in\{(\vec{x},r)|r\leq f(\vec{x})\}$, then we can find a sequence of points, $\{ \vec{x}_k\}\subset\mathbb{R}^n$ so that $\vec{x}_k\rightarrow\vec{x}$.  Thus, for $k$ large enough, $|(\vec{x}_k,r)-(\vec{x},r)|$ is small.  Thus, $cl(\mathcal{E})=cl(\mbox{epi}f)$.  A similar argument can be used to give the second equation in (\ref{touchingf}) for $f$ lower semicontinuous.  Using these two results, we get that (\ref{touchingf}) holds for any function $f$.
 
 \begin{figure}[htbp!]
 \centering
 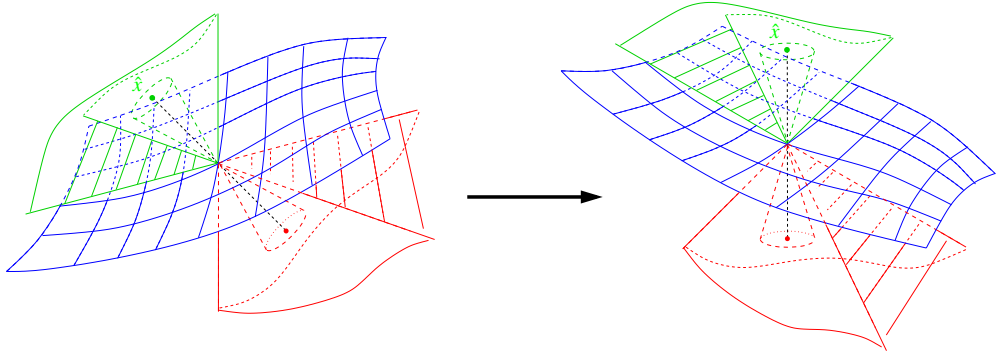
 \caption{Rotating the graph of $f$ so that the line segment from $y$ to $\hat{x}$ becomes vertical}\label{rotating}
 \end{figure}
 Picking $\hat{x}\in \underline{K}$ so that $B_\delta(\hat{x})\subset \underline{K}$ and rotating so that $\hat{x}$ becomes vertical (see Figure \ref{rotating}), the piece of $\partial\mbox{epi} f$ in any $B_\delta(y)$, $y\in\mbox{epi}f$ will be a Lipschitz graph with Lipschitz constant no more than $\left(\sqrt{||\hat{x}||^2+\delta^2}\right)/\delta$. This implies that $\mu(\partial\mbox{epi f})<\infty$ in any ball, that is, for $y\in\partial\mbox{epi}$, $\mu(\partial\mbox{epi} f\cap B(y,R))<\infty$ for any $R$.
 
\begin{theorem}\label{BVtheorem}
If $f$ is $K$ monotone and bounded, and $K$ has non-empty interior then $f\in BV$.
\end{theorem}
\begin{proof}
First, the slicing theorem from \cite{KrantzParks} gives us that 
\begin{equation}
\int_{-\infty}^\infty(\partial\mbox{epo}f)_t\;dt<\mu(\partial\mbox{epo}f), 
\end{equation}
where $(\partial\mbox{epo}f)_t=\partial\{x|f(x)\geq t\}$.
So we have that 
\begin{equation}\label{coareaint}
\int_{-\infty}^\infty(\partial\mbox{epo}f)_t\;dt<\infty.
\end{equation}
Using the coarea formula for BV functions from \cite{EvansGariepy}, we get that (\ref{coareaint}) implies that $f\in BV$.
\end{proof}
\end{subsubsection}

%%%%%%%%%%%%%%%%%%%%%%%%%%%%%%%%%%%%%%%%%%%%%%%
%%%Examples of cone monotone
%%%%%%%%%%%%%%%%%%%%%%%%%%%%%%%%%%%%%%%%%%%%%%%

\begin{subsubsection}{Examples of Cone Monotone Functions}
We now consider some examples of $K$ monotone functions.  
 
Suppose $K$ is a ray so that $K$ has empty interior. Then for $f$ to be $K$ monotone all we need is for $f$ monotone on all lines parallel to $K$, that is monotone in the positive direction of $K$.  Therefore, $f$ need not even be measurable.
%\begin{figure}[htbp!]
%\includegraphics[width=3in]{randfct2.png}
%\caption{$f(x,y)=rand(y)$}\label{randfct} 
%\end{figure}
\begin{exm}\label{RandConeMono}
Let $f(\cdot,y)=rand(y)$,
%(see Figure \ref{randfct})
where $rand(y)$ assigns a particular random number to each value $y$. This function need not be measurable, but is $K$ monotone with $K=\left\{\alpha\left[\begin{array}{c} 1\\ 0\end{array}\right]|\alpha>0\right\}$.
\end{exm}

\begin{figure}[htbp]
\centering
\includegraphics[width=2in]{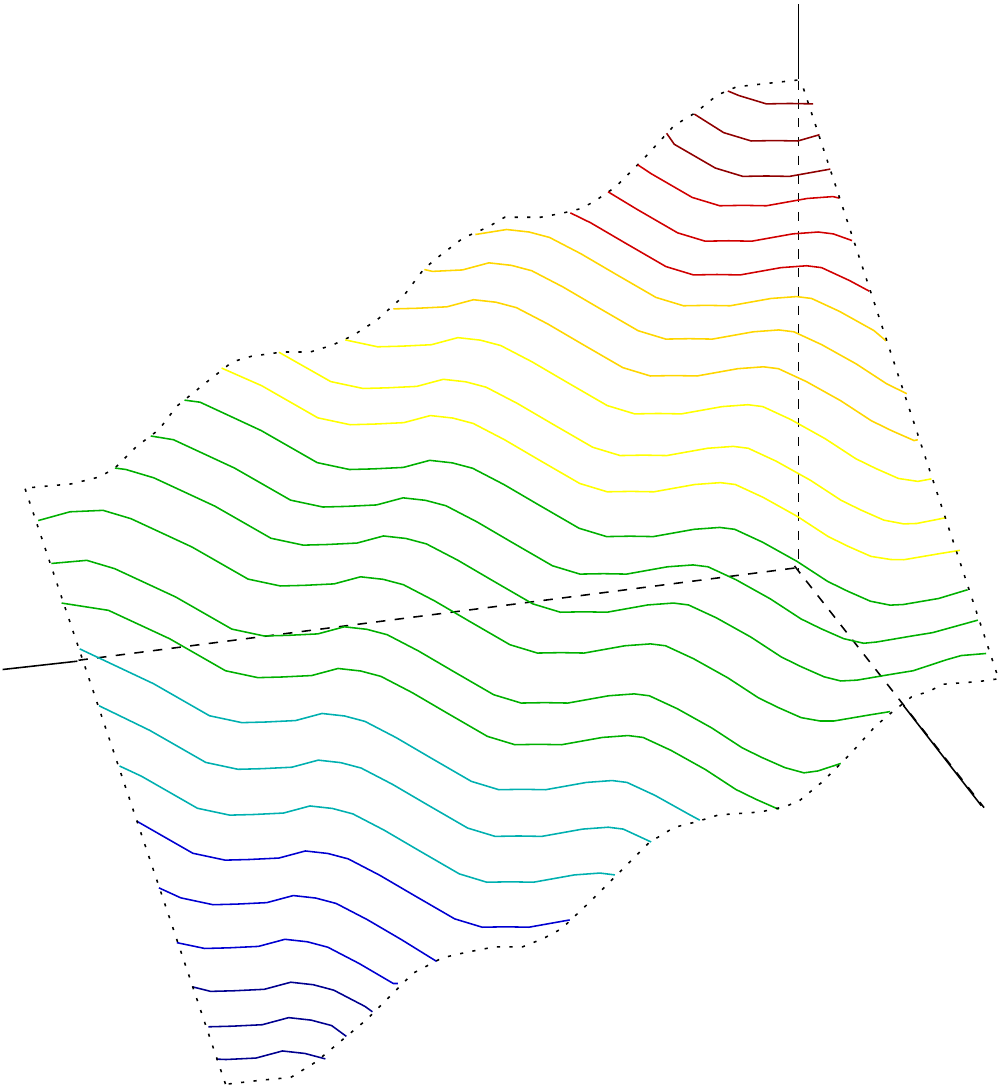}
\caption{$f(x,y)=\sin(x)+x+y$.}\label{oscCM}
\end{figure}
 \begin{exm}\label{OscConeMono}
An example of a $K$ monotone function with the cone, $K$ having nonempty interior is a function that oscillates, but is sloped upward (see Figure \ref{oscCM}).  More specifically, the function $f(x,y)=\sin(x)+x+y$ is $K$ monotone.  We can see this by noticing that $f$ is increasing in the cone
 $K=\{(v_1,v_2)| v_1>0, v_2>0\}$.
\end{exm}

\begin{remark}Notice in this example that $f$ has oscillatory behavior.  Yet, $f$ is still cone monotone. 
Notice also that if an oscillating function is tipped enough the result is $K$ monotone.  The more tipped the more oscillations possible and still be able to maintain $K$ monotonicity.
\end{remark}

\begin{exm}\label{ConeNotVG} Some cone monotone functions are monotone in no other sense.  An example of a function, $f:\mathbb{R}^2\rightarrow\mathbb{R}$, that is Cone monotone, but not Vodopyanov Goldstein monotone is a function whose graph is a paraboloid.  At each point $x$, that is not the vertex of the paraboloid, we find the normal to the level set $\{f=f(x)\}$.  We see the half space determined by this  normal is a cone in which $f$ increases from $f(x)$.  At the vertex of the paraboloid, we see that all of $\mathbb{R}^2$ is the cone in which $f$ increases.
\end{exm}

\begin{exm}\label{VGNotCone} Not all Vodopyanov-Goldstein Monotone functions are cone monotone.  An example of a function that is Vodopyanov-Goldstein Monotone, but is not cone monotone can be constructed with inspiration from Example \ref{VGNotL}.  Level sets of this function are drawn in Figure \ref{VGnotConeFig}.  Here we see that the darkest blue level (minimum) set turns too much to be Cone monotone.  We see this at the point $y$.  At this point, there is no cone so that all points inside the cone have function value larger than $f(y)$ since any cone will cross the dark blue level set at another point.
\end{exm}
\begin{figure}[htbp!]
\centering
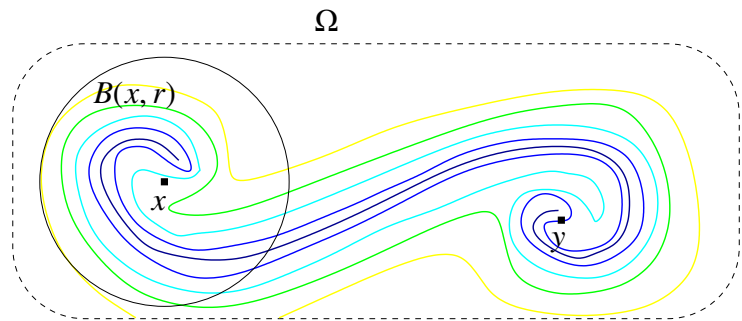
\caption{A function that is Vodopyanov-Goldstein monotone, but is not Cone monotone.}\label{VGnotConeFig}
\end{figure}

\begin{exm}\label{LebNotCone}
We can create a function, $f$ that is Lebesgue Monotone, but is not Cone monotone.  In this case, we need a level set that turns too much, but the level sets extend to the boundary of $\Omega$.  We see such a function in Figure \ref{LebNotConeFig}.  Let dark blue represent a minimum.  Then at the point $y$, there is no cone that so that every point in the cone has function value larger than $f(y)$ since every cone will cross the dark blue level set.
\end{exm}
\begin{figure}[htbp!]
\centering
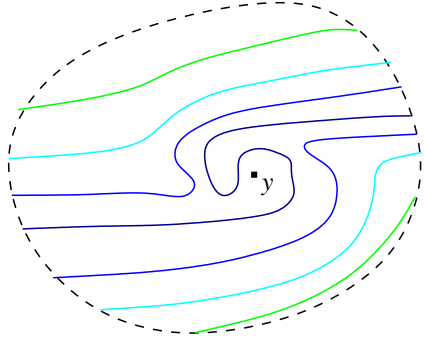
\caption{A function that is Lebesgue monotone, but is not Cone monotone.}\label{LebNotConeFig}
\end{figure}

Now if the domain has dimension higher than 2 and $K$ is convex and has empty interior, but is not just a ray, then we can look at slices of the domain (see Figure \ref{NonemptyInterior}).  We can see that on each slice of the domain, the function still satisfies Theorems \ref{nolocextr} and \ref{BVtheorem}.  But, we also see that the behavior of the function from slice to slice is independent. This is the same behavior as we see when the function is defined on a 2-dimensional domain and $K$ is a ray.  That is, from line to line, the function behavior is independent (see Example \ref{RandConeMono}). We can also see an example of the extended cones for a $K$ monotone function where $K$ is a ray, in Figure \ref{poscodim2}.
\begin{figure}[htbp!]
\includegraphics[width=3.5in]{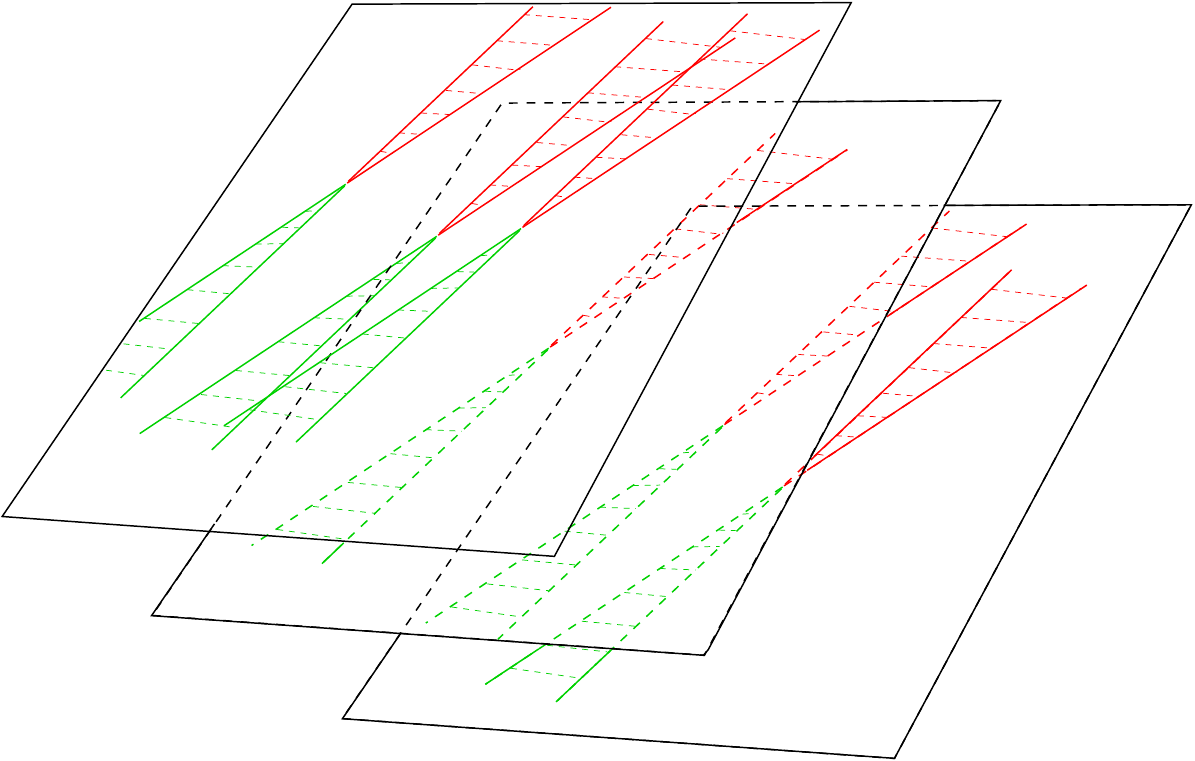}
\caption{Cones with empty interior in a 3D domain that are not just a ray.}\label{NonemptyInterior}
\end{figure}

\begin{figure}[htbp!]
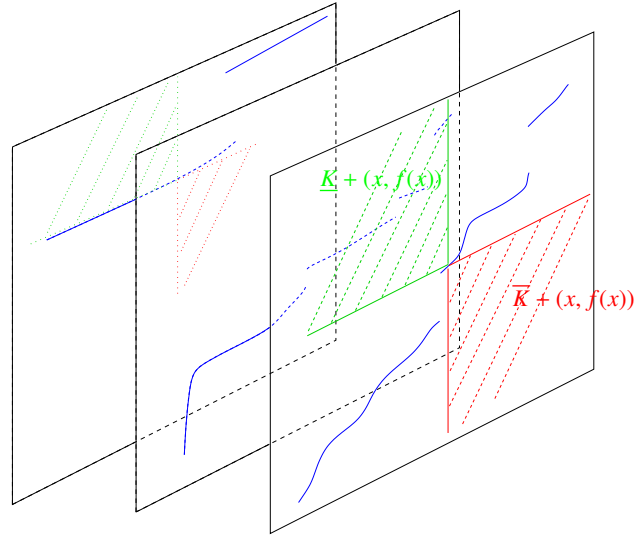
\caption{The extended cones are
shown pinching the graphs of the functions, shown in blue.
The key point is that the blue curve in each leaf of
the foliation is independent of every other graph.}\label{poscodim2}
\end{figure}

If $K$ is a closed half space then $f$ has level sets that are hyperplanes parallel to $\partial K$ and $f$ is one dimensional monotone.
\end{subsubsection}
%%%%%%%%%%%%%%%%%%%%%%%%%%%%%%%%%%%%%%%%%%%%%%%
%%%Construction of K monotone functions
%%%%%%%%%%%%%%%%%%%%%%%%%%%%%%%%%%%%%%%%%%%%%%%
\begin{subsubsection}{Construction of $K$ monotone functions}\label{construction}
Recall from (\ref{touchingf}) that if $f$ is $K$ monotone, we have
\begin{eqnarray}\label{closureEpi}
cl(\mbox{epi}f)=\overline{\bigcup_{\vec{x}\in\mathbb{R}^n}(\vec{x},f(\vec{x}))+\underline{K}}\nonumber\\
cl(\mbox{epo}f)=\overline{\bigcup_{\vec{x}\in\mathbb{R}^n}(\vec{x},f(\vec{x}))+\overline{K}}
\end{eqnarray}

We can also construct a $K$ monotone function by taking arbitrary unions of the sets $(\vec{x},x_{n+1})+\underline{K}$.  By construction the boundary of this set is then the graph of the epigraph (and of the epograph) of a $K$ monotone function.
\end{subsubsection}
%%%%%%%%%%%%%%%%%%%%%%%%%%%%%%%%%%%%%%%%%%%%%%%
%%%Bounds on TV Norm
%%%%%%%%%%%%%%%%%%%%%%%%%%%%%%%%%%%%%%%%%%%%%%%

\begin{subsubsection}{Bounds on TV Norm}
In this section, we find a bound on the total variation of $K$ monotone functions.  To do this we use the idea of a tipped graph introduced in Subsection \ref{Characteristics}.

Suppose $f<C$ on $\mathbb{R}^n$.  Then $f\left|_{B(0,R)\subset\mathbb{R}^n}\right.$ has a graph that is 
contained in $B\left(0,\sqrt{R^2+C^2}\right)\subset\mathbb{R}^{n+1}$.  Assuming that the tipped Lipschitz constant 
is $L \left(\leq\frac{\sqrt{||\vec{x}||^2+\delta^2}}{\delta}\right)$, we get that the amount of $\left.\partial\mbox{epi} f\right|_{f|_{B(0,R)}}$ in $B\left(0,\sqrt{R^2+C^2}\right)$ is bounded above by $\alpha(n)\left(\sqrt{R^2+C^2}\right)^n\sqrt{1+L^2}$, where $\alpha(n)$ is the volume of the $n$ dimensional unit ball.

Using the coarea formula discussed above, we get an upper bound on the total variation of a function that is $K$ monotone as follows.
\begin{equation}
TV_{B(0,R)}(f)=\int_{B(0,R)}|\nabla u|\;dx\leq\mu(\partial\mbox{epi}(f(B(0,R))))\leq\alpha(n)\left(\sqrt{R^2+C^2}\right)^n\sqrt{1+L^2}.
\end{equation}
%**I need more discussion here**
%Discuss the whole integral variation...K talked about this

\end{subsubsection}

%%%%%%%%%%%%%%%%%%%%%%%%%%%%%%%%%%%%%%%%%%%%%%%%%
%% Normal Monotone
%%%%%%%%%%%%%%%%%%%%%%%%%%%%%%%%%%%%%%%%%%%%%%%%%

\begin{subsection}{Normal Monotone}
Motivated by the nondecreasing (or nonincreasing) behavior of monotone functions with domain in $\mathbb{R}$, we introduce a specific case of cone monotone.  We consider a notion of monotonicity for functions whose domain is in $\Omega\subset\mathbb{R}^2$  by requiring that a monotone function be nondecreasing (or nonincreasing) in a direction normal to the level sets of the function.

First, we introduce a few definitions.
\begin{definition}
We say that a vector $v$ is tangent to a set $X$ at a point $x\in X$ if there is a sequence $\{x_k\}\subset X$ with $x_k\rightarrow x$ and a sequence $\{t_k\}\subset \mathbb{R}$ with  $t_k\searrow 0$ so that 
\begin{equation}
\lim_{k\rightarrow\infty}\frac{x_k-x}{t_k}=v.
\end{equation}
The set of all tangents to the set $X$ at the point $x\in X$ is the tangent cone and denote it by $T_X(x)$.
\end{definition} 
\begin{definition}
We say that a vector $n(x)$ is normal to a set $X$ at a point $x\in X$ if for every vector $v\in T_X(x)$ we have that $n(x)\cdot v\leq 0$.
\end{definition} 
\begin{definition}\label{NormalMono}
We say that a function $f :\mathbb{R}^2\rightarrow\mathbb{R}$ is (strictly) Normal monotone if for every $c\in\mathbb{R}$ and every $x$ on the boundary of the level set $\{f =c\}$ the 1-dimensional functions $\gamma\mapsto f (x+\gamma n(x))$ are (strictly) monotone for every vector, $n(x)$, normal to the level set $\{f =c\}$ at $x$.
\end{definition}
\begin{remark}
The definition for normal monotone requires
 that the function be monotone along the entire
 intersection of a one dimensional line and the the
 domain of f. In the case of cone monotone, we require
 only monotonicity in the direction of the positive cone
 while in the case of K monotone, the fact that we can look
 forwards and backwards to get non-decreasing and non-increasing
 behavior follows from the invariance of K, not the definition
 of cone monotone.
\end{remark}
\begin{remark}
Notice also that Example \ref{OscConeMono} is not normal monotone.
\end{remark}
%Now, suppose $K$ is a convex cone with nonempty interior.  Then 
%$f$ is $K$ monotone and bounded on compact sets implies that the graph of $f$ is a tipped Lipschitz graph which implies that $ f\in BV$.

\begin{remark} A smooth function that is normal monotone is cone monotone for any cone valued function $K(x)\subset N(x)\;\forall x$.
\end{remark}

We now explore this definition with a few examples.  
\begin{exm}
One can easily verify that a function whose graph is a non-horizontal plane is strictly normal monotone.  This is desirable since a 1D function whose graph is a line is strictly monotone (assuming it is not constant).  
\end{exm}

\begin{exm}
A function whose graph is a parabola is not monotone in 1D so neither should a paraboloid be normal monotone.  One can easily verify this to be the case.
\end{exm}
%\begin{figure}[!htbp]
%\centering
%\includegraphics{paraboloid.pdf}
%\caption{$f(x,y)=x^2+y^2$}\label{parabMono}
%\end{figure}

\begin{exm}\label{LebNotKNotNormal}
If we extend a nonmonotone 1D function to 2D, we should get a function that is not normal monotone. An example of such a function is the function $f(x,y)=x^3-x$.  Notice, this function is Lebesgue monotone, but neither $K$ nor Normal monotone.
\end{exm}

\begin{exm}\label{KNotNormal}
In Figure \ref{KNotNormalFig}, we show a function whose level sets are very oscillatory so that it is not normal monotone, while still being $K$ monotone.
\end{exm}
\begin{figure}[!htbp]
\centering
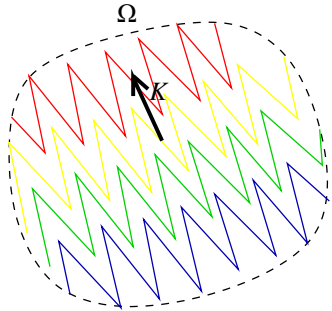
\caption{A function that is $K$ monotone, but not Normal monotone.}\label{KNotNormalFig}
\end{figure}

\begin{exm}\label{NormalNotK}
In Figure \ref{NormalNotKFig}, we see that if $\Omega$ is not convex, then we can construct a function that is not $K$ monotone, but is Normal monotone.  In this example, the function increases in a counter clockwise direction.  This function is Normal monotone. We can see that it is not $K$ monotone since at the point $x$ any direction pointing to the north and west of the level line is a direction of ascent.  But, at the point $y$, these directions are directions of descent.  So the only cone of ascent at both $x$ and $y$ must be along the line parallel to their level curves.  But, we see that at other points, this is not a cone of ascent.  Thus, this function cannot be $K$ monotone.
\end{exm}
\begin{figure}[!htbp]
\centering
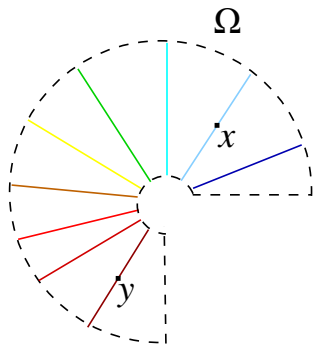
\caption{A function that is not $K$ monotone, but is Normal monotone.}\label{NormalNotKFig}
\end{figure}

The next theorem tells us that a normal monotone function is also Lebesgue monotone.
%%%%%%%%%%%%%%%%%%%%%%%%%%%%%%%%%%%%%%%%%%%%
% Continuous and Normal Monotone ===> L-Monotone
%%%%%%%%%%%%%%%%%%%%%%%%%%%%%%%%%%%%%%%%%%%%
\begin{theorem}\label{NMimpliesLM}
Let $\Omega\subset\mathbb{R}^2$ be a bounded domain and let $f:\Omega\rightarrow\mathbb{R}$ be a continuous, normal monotone function then $f$ is also Lebesgue monotone.
\end{theorem}
\begin{proof}  We prove the contrapositive.  Suppose $f$ is not Lebesgue Monotone.  Then there exists a set $\Omega^\prime$ so that
\begin{equation*}
\inf_{x\in\Omega^\prime}f(x)<\inf_{x\in\partial\Omega^\prime}f(x).
\end{equation*}
We want to show that $f$ is not normal monotone.
Let us then define the nonempty set $M\subset\Omega^\prime$ to be the set  where $f$ attains a local  minimum, at every point in $M$.  That is,
\begin{equation*}
M=\left\{x\in\Omega^\prime|f(x)=\inf_{x\in\Omega^\prime}f(x)\right\}.
\end{equation*}
Let $(x_0,y_0)\in\partial M$ and let $n(x_0,y_0)$ be a normal at  $(x_0,y_0)$ to $M$.  We know then that $\gamma\mapsto f((x_0,y_0)+\gamma n(x_0,y_0))$ is not monotone since $f$ has a local minimum on $M$.  Thus $f$ is not normal monotone.
\end{proof}
\begin{remark}
This theorem gives us that a function that is normal monotone is also weakly monotone.
\end{remark}
\end{subsection}

\end{section}
%%%%%%%%%%%%%%%%%%%%%%%%%%%%%%%%%%%%%%%%%%%%
% Conclusion
%%%%%%%%%%%%%%%%%%%%%%%%%%%%%%%%%%%%%%%%%%%%
\begin{section}{Conclusion}
In this paper we explored current and new definitions of monotone.  For continuous functions, we compared several definitions of monotonicity, in higher dimensions. How these sets of functions are related is represented in the following Venn diagram.
\begin{figure}[htbp!]
\centering
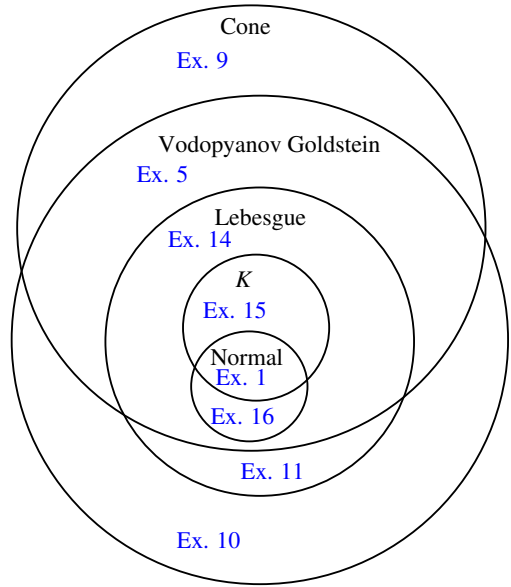
\caption{Types of monotonicity in higher dimensions and how they compare for a continuous function.}
\end{figure}
\end{section}

We also showed how to construct $K$ monotone functions.  We show that bounded $K$ monotone functions are BV and we find a bound on the total variation of these functions.
%%%%%%%% Bibliography %%%%%%%%

\bibliographystyle{plain}
\bibliography{MonotoneComp.bib}

Addendum: After the publication of this article
in Abstract and Applied Analysis we discovered a
1993 paper by Clarke et al. that contained a couple
of our results. The full reference is:
\emph{Subgradient criteria for monotonicity, the
Lipschitz condition, and convexity} by FH Clarke, RJ Stern,
and PR Wolenski, in Canadian Journal of Mathematics, Volume 45
number 6, 1993
\end{document}